\documentclass[12pt]{article}
\usepackage{amsmath,amssymb,amsthm,comment}

\newtheorem{theorem}{Theorem}[section]
\newtheorem{thmy}{Theorem}

\newtheorem{lemma}[theorem]{Lemma}
\newtheorem{corollary}[theorem]{Corollary}

\def\barr{\begin{array}}
\def\earr{\end{array}}

\title{Finite groups with many normalizers}
\author{Marius T\u arn\u auceanu}
\date{August 4, 2024}

\begin{document}

\maketitle

\begin{abstract}
A group $G$ is said to have dense normalizers if each non-empty open interval in its subgroup lattice $L(G)$ contains the normalizer of a certain subgroup of $G$. In this note, we find all finite groups satisfying this property. We also classify the finite groups in which $k$ subgroups are not normalizers, for $k=1,2,3,4$.
\end{abstract}

{\small
\noindent
{\bf MSC2000\,:} Primary 20D60; Secondary 20D30, 20E99.

\noindent
{\bf Key words\,:} normalizers, density, finite groups.}

\section{Introduction}

Let $\mathcal{X}$ be a property pertaining to subgroups of a group. We say that a group $G$ has \textit{dense $\mathcal{X}$-subgroups} if for each pair $(H,K)$ of subgroups of $G$ such that $H<K$ and $H$ is not maximal in $K$, there exists a $\mathcal{X}$-subgroup $X$ of $G$ such that $H<X<K$. Groups with dense $\mathcal{X}$-subgroups have been studied for many properties $\mathcal{X}$: being a normal subgroup \cite{9}, a pronormal subgroup \cite{13}, a normal-by-finite subgroup \cite{4}, a nearly normal subgroup \cite{5} or a subnormal subgroup \cite{6}.

Our first result deals with the property of being the normalizer of a certain subgroup.

\begin{theorem}
A finite group $G$ has dense normalizers if and only if $G$ is either cyclic of prime order or non-abelian of order $pq$ for some primes $p$ and $q$. 
\end{theorem}

Next we study finite groups with many normalizers in another sense. Let $G$ be a finite group, $L(G)$ be the subgroup lattice of $G$ and $N_G=\{N_G(H)\mid H\in L(G)\}$. Obviously, we have $|N_G|=|L(G)|$ if and only if $G$ is trivial.\newpage \noindent Our second result describes finite groups $G$ satisfying $|N_G|=|L(G)|-k$, where $k=1,2,3$. Finite groups $G$ such that $|N_G|=|L(G)|-4$ are also determined modulo ZM-groups.

\begin{theorem}
For a finite group $G$, we have:
\begin{itemize}
\item[{\rm a)}] $|N_G|=|L(G)|-1$ if and only if $G\cong\mathbb{Z}_p$ with $p$ prime.
\item[{\rm b)}] $|N_G|=|L(G)|-2$ if and only if $G\cong\mathbb{Z}_{p^2}$ with $p$ prime or $G\cong\mathbb{Z}_p\rtimes\mathbb{Z}_q$ with $p,q$ primes satisfying $q\mid p-1$.
\item[{\rm c)}] $|N_G|=|L(G)|-3$ if and only if $G\cong\mathbb{Z}_{p^3}$ with $p$ prime or $G\cong\mathbb{Z}_{pq}$ with $p,q$ distinct primes or $G\cong\mathbb{Z}_{p^2}\rtimes\mathbb{Z}_q$ with $p,q$ primes satisfying $q\mid p-1$.
\item[{\rm d)}] If $G$ is not a ZM-group, then $|N_G|=|L(G)|-4$ if and only if $G\cong\mathbb{Z}_{p^4}$ with $p$ prime or $G\cong\mathbb{Z}_2\times\mathbb{Z}_2$ or $G\cong A_4$.
\end{itemize} 
\end{theorem}

We note that finite groups $G$ with few normalizers, i.e. with small number $|N_G|$, have been studied in \cite{8,10,14,15}. Also, we remark that item d) in Theorem 1.2 leads to the following characterization of $A_4$.

\begin{corollary}
$A_4$ is the unique finite non-abelian group $G$ satisfying $|N_G|=|L(G)|-4$ and having at least one non-cyclic Sylow subgroup.
\end{corollary}

We recall that a \textit{ZM-group} is a finite non-abelian group with all Sylow subgroups cyclic. By \cite{7}, such a
group is of type
\begin{equation}
{\rm ZM}(m,n,r)=\langle a, b \mid a^m = b^n = 1,\, b^{-1} a b = a^r\rangle,\nonumber
\end{equation}where the triple $(m,n,r)$ satisfies the conditions
\begin{equation}
{\rm gcd}(m,n) = {\rm gcd}(m, r-1) = 1 \mbox{ and } r^n \equiv 1 \hspace{1mm}({\rm mod}\hspace{1mm}m).\nonumber
\end{equation}It is clear that $|{\rm ZM}(m,n,r)|=mn$ and $Z({\rm ZM}(m,n,r))=\langle b^d\rangle$, where $d$ is the multiplicative order of
$r$ modulo $m$, i.e.
\begin{equation}
d=o_m(r)={\rm min}\{k\in\mathbb{N}^* \mid r^k\equiv 1 \hspace{1mm}({\rm mod} \hspace{1mm}m)\}.\nonumber
\end{equation}The subgroups of ${\rm ZM}(m,n,r)$ have been completely described
in \cite{1}. Set
\begin{equation}
L=\left\{(m_1,n_1,s)\in\mathbb{N}^3 \,:\, m_1|m, n_1|n,s<m_1, m_1|s\frac{r^n-1}{r^{n_1}-1}\right\}.\nonumber
\end{equation}Then there is a bijection between $L$ and the subgroup lattice
$L({\rm ZM}(m,n,r))$ of ${\rm ZM}(m,n,r)$, namely the function
that maps a triple $(m_1,n_1,s)\in L$ into the subgroup $H_{(m_1,n_1,s)}$ defined by\newpage
\begin{equation}
H_{(m_1,n_1,s)}=\bigcup_{k=1}^{\frac{n}{n_1}}\alpha(n_1,s)^k\langle a^{m_1}\rangle=\langle a^{m_1},\alpha(n_1, s)\rangle,\nonumber
\end{equation}where $\alpha(x, y)=b^xa^y$, for all $0\leq x<n$ and $0\leq y<m$. We remark that $|H_{(m_1,n_1,s)}|=\frac{mn}{m_1n_1}$\,, for any $s$ satisfying $(m_1,n_1,s)\in L$.
\bigskip

For the proof of Theorem 1.2, we need the following well-known result on $p$-groups (see e.g. (4.4) of \cite{12}, II).

\begin{thmy}
The following conditions on a finite $p$-group $G$ are equivalent:
\begin{itemize} 
\item[{\rm a)}] Every abelian subgroup of $G$ is cyclic.
\item[{\rm b)}] $G$ has a unique subgroup of order $p$.
\item[{\rm c)}] $G$ is either cyclic or a generalized quaternion $2$-group.
\end{itemize}
\end{thmy}

By a \textit{generalized quaternion $2$-group} we will understand a group of order $2^n$ for some positive integer $n\geq 3$, defined by
\begin{equation}
Q_{2^n}=\langle a,b \mid a^{2^{n-2}}= b^2, a^{2^{n-1}}=1, b^{-1}ab=a^{-1}\rangle.\nonumber
\end{equation}

We also need the classification of finite groups with all non-trivial elements of prime order (see e.g. \cite{2,3}). 

\begin{thmy}
Let $G$ be a finite group having all non-trivial elements of prime order. Then:
\begin{itemize} 
\item[{\rm a)}] $G$ is nilpotent if and only if $G$ is a $p$-group of exponent $p$.
\item[{\rm b)}] $G$ is solvable and non-nilpotent if and only if $G$ is a Frobenius group with kernel $P\in{\rm Syl}_p(G)$ with $P$ a $p$-group of exponent $p$ and complement $Q\in{\rm Syl}_q(G)$ with $|Q|=q$. Moreover, if $|G|=p^nq$ then $G$ has a chief series $G=G_0>P=G_1>G_2>...>G_k>G_{k+1}=1$ such that for every $1\leq i\leq k$ one has $G_i/G_{i+1}\leq Z(P/G_{i+1})$, $Q$ acts irreducibly on $G_i/G_{i+1}$ and $|G_i/G_{i+1}|=p^b$, where $b$ is the exponent of $p\,\, ({\rm mod}\,\, q)$.
\item[{\rm c)}] $G$ is non-solvable if and only if $G\cong A_5$.
\end{itemize}
\end{thmy}

Most of our notation is standard and will not be repeated here. Basic definitions and results on groups can be found in \cite{7,12}. For subgroup lattice concepts we refer the reader to \cite{11}.

\section{Proofs of the main results}

\noindent{\bf Proof of Theorem 1.1.} Assume that $G$ has dense normalizers. 

First of all, we will prove that $G$ is of square-free order. Let $|G|=p_1^{n_1}\cdots p_k^{n_k}$ and assume that $n_1\geq 2$. Then $G$ contains a subgroup $P_1$ of order $p_1^2$. By hypothesis, there is $H\leq G$ such that $1<N_G(H)<P_1$. It follows that $|N_G(H)|=p_1$. Since $H\subseteq N_G(H)$, we have either $H=1$ implying that $G=N_G(H)$ is of order $p_1$ - a contradiction, or $H=N_G(H)$ implying that $P_1\subseteq H$ - a contradiction. Thus $n_1=...=n_k=1$.

Assume next that $k\geq 3$. Clearly, $G$ cannot be cyclic. Then $G$ is a non-trivial semidirect product of type $\mathbb{Z}_{p_1\cdots p_r}\rtimes\mathbb{Z}_{p_{r+1}\cdots p_k}$. Since $r\geq 2$ or $k-r\geq 2$, $G$ possesses a cyclic subgroup $K$ of order $p_ip_j$ for some $i,j\in\{1,...,k\}$. Similarly, we get that the interval $(1,K)$ of $L(G)$ contains no normalizer, a contradiction. Thus $k\leq 2$ and $G$ is either cyclic of prime order or non-abelian of order a product of two primes.

The converse is obvious, completing the proof.\hspace{46mm}$\square$
\smallskip

The following lemma will be useful for the proof of Theorem 1.2.

\begin{lemma}
Let $G$ be a finite group such that $|L(G)\setminus N_G|\leq 3$. Then $G$ is of one of the following types:
\begin{itemize} 
\item[{\rm a)}] A cyclic group $\mathbb{Z}_{p^a}$ with $p$ prime and $a\leq 3$, or $\mathbb{Z}_{pq}$ with $p,q$ distinct primes.
\item[{\rm b)}] A ZM-group ${\rm ZM}(m,n,r)$ with $\tau(m)+\tau(n)\leq 5$.
\end{itemize}
\end{lemma}

\begin{proof}
First of all, we observe that $G$ cannot have subgroups of type $\mathbb{Z}_p\times\mathbb{Z}_p$ for no prime $p$. Indeed, if $H$ would be such a subgroup of $G$, then all $p+2$ proper subgroups of $H$ are not normalizers and $p+2\geq 4$, contradicting the hypothesis. Thus all abelian $p$-subgroups of $G$ are cyclic. By Theorem A, it follows that the Sylow subgroups of $G$ are either cyclic or generalized quaternion $2$-groups. If $G$ has a Sylow $2$-subgroup of type $Q_{2^n}$, then there is $Q\leq G$ such that $Q\cong Q_8$. We easily get that all $5$ proper subgroups of $Q$ are not normalizers, a contradiction. Consequently, the Sylow subgroups of $G$ are cyclic, implying that $G$ is cyclic or a ZM-group.

If $G\cong\mathbb{Z}_n$, then the condition $|L(G)\setminus N_G|\leq 3$ means $\tau(n)\leq 4$ and we get $n=p^a$ with $p$ prime and $a\leq 3$, or $n=pq$ with $p,q$ distinct primes. If $G\cong {\rm ZM}(m,n,r)$, then $G$ has a normal cyclic subgroup $G_1$ of order $m$ and a cyclic subgroup $G_2$ of order $n$. Since all subgroups of $G_1$, as well as all proper subgroups of $G_2$, are not normalizers, we infer that $$\tau(m)+\tau(n)-2\leq |L(G)\setminus N_G|\leq 3$$and so $\tau(m)+\tau(n)\leq 5$, as desired.
\end{proof}

We are now able to prove Theorem 1.2..
\bigskip

\noindent{\bf Proof of Theorem 1.2.} Items a)-c) follow immediately from Lemma 2.1. For item d), assume that $|N_G|=|L(G)|-4$ but $G$ is not a ZM-group. Then $G$ is either cyclic, say $G\cong\mathbb{Z}_n$, or has a non-cyclic Sylow $p$-subgroup $P$. In the first case we get $\tau(n)=5$, i.e. $n=p^4$ for some prime $p$, while in the second case we observe that $P$ cannot be a generalized quaternion $2$-group.\footnote{As we have seen above, such a group has at least $5$ subgroups which are not nor\-ma\-li\-zers.} Then Theorem A shows that $P$ contains a subgroup $H$ of type $\mathbb{Z}_p\times\mathbb{Z}_p$. Since all proper subgroups of $H$ are not normalizers, we have $p+2\leq 4$, i.e. $p=2$. Also, it is clear that $H$ is normal in $G$. If $H\neq P$, then there is $K\leq P$ with $|K|=8$ and $H\subset K$. Since $\mathbb{Z}_2\times\mathbb{Z}_4$ is the unique group of order $8$ having only one subgroup of type $\mathbb{Z}_2\times\mathbb{Z}_2$, we infer that $K\cong\mathbb{Z}_2\times\mathbb{Z}_4$. Let $L\leq K$ such that $|L|=4$ and $L\neq H$. By hypothesis, $L$ is a normalizer, say $L=N_G(M)$ with $M\leq G$. Then $M\subseteq L\subset K$ and since $K$ normalizes all its subgroups we get $K\subseteq L$, a contradiction. Thus $H=P$, that is $G$ has a normal Sylow $2$-subgroup of type $\mathbb{Z}_2\times\mathbb{Z}_2$.

If $G=P\cong\mathbb{Z}_2\times\mathbb{Z}_2$, we are done. Assume next that $|G|=4p_1^{n_1}\cdots p_k^{n_k}$ for some odd primes $p_1,...,p_k$. As in the proof of Theorem 1.1, we get $n_1=...=n_k=1$. Let $i\in\{1,...,k\}$ and $P_i$ be a Sylow $p_i$-subgroup of $G$. Then $P_i$ is a normalizer, more precisely $P_i=N_G(P_i)$. This shows that $P_i$ is not properly contained in any abelian subgroup of $G$, implying that $G$ has no element of composite order. In other words, all non-trivial elements of $G$ have prime order and therefore $G$ is one of the groups described in Theorem B. Since $G$ is not a $p$-group and $A_5$ does not satisfy the hypothesis, we infer that $G$ is a Frobenius group with kernel $\mathbb{Z}_2\times\mathbb{Z}_2$ and complement of order $p_1$. It follows that $p_1=3$ and $G\cong A_4$, completing the proof.\hspace{50mm}$\square$
\smallskip

Finally, we note that ZM-groups $G\cong {\rm ZM}(m,n,r)$ satisfying $|N_G|=|L(G)|-4$ can be classified according to the inequality $\tau(m)+\tau(n)\leq 6$. The dihedral groups $D_{30}$ and $D_{54}$ are examples of such groups.\vspace{2mm}\newpage

\bigskip{\bf Funding.} The author did not receive support from any organization for the submitted work.

\bigskip{\bf Conflicts of interests.} The author declares that he has no conflict of interest.

\bigskip{\bf Data availability statement.} My manuscript has no associated data.

\vspace*{5ex}\small

\hfill
\begin{minipage}[t]{5cm}
Marius T\u arn\u auceanu \\
Faculty of  Mathematics \\
``Al.I. Cuza'' University \\
Ia\c si, Romania \\
e-mail: {\tt tarnauc@uaic.ro}
\end{minipage}

\end{document}